\documentclass[11pt]{article}
\usepackage{latexsym,amsmath,amsthm,amssymb,color}
\usepackage{mathrsfs,wasysym,graphicx,lscape,ulem}
\usepackage{esint,color}

\begin{document}

\newtheorem{theorem}{\indent Theorem}[section]
\newtheorem{proposition}[theorem]{\indent Proposition}
\newtheorem{definition}[theorem]{\indent Definition}
\newtheorem{lemma}[theorem]{\indent Lemma}
\newtheorem{remark}[theorem]{\indent Remark}
\newtheorem{corollary}[theorem]{\indent Corollary}

\begin{center}
    {\large \bf Pullback Attractors for a Critical Degenerate Wave Equation\\ with Time-dependent Damping}
\vspace{0.5cm}\\{Dandan Li$^1$, Qingquan Chang$^2$, Chunyou Sun$^{2*}$}\\
{\small 1.School of Mathematics and Statistics, Huazhong University of Science\\ and Technology,
 Wuhan, 430074, PR China}\\
{\small 2.School of Mathematics and Statistics, Lanzhou University\\
Lanzhou, 730000, PR China }\\

\end{center}

\renewcommand{\theequation}{\arabic{section}.\arabic{equation}}
\numberwithin{equation}{section}

\begin{abstract} The aim of this paper is to analyze the long-time dynamical behavior of the solution for a degenerate wave equation with time-dependent damping term $\partial_{tt}u + \beta(t)\partial_tu = \mathcal{L}u(x,t) + f(u)$ on a bounded domain $\Omega\subset\mathbb{R}^N$ with Dirichlet boundary conditions. Under some restrictions on $\beta(t)$ and critical growth restrictions on the nonlinear term $f$, we will prove the local and global well-posedness of the solution and derive the existence of a pullback attractor for the process associated with the degenerate damped hyperbolic problem.

\noindent

\end{abstract}
\footnote[0]{\hspace*{-7.4mm}
$^*$Corresponding author.
Emails: lidd2008@lzu.edu.cn(D.Li); sunchy@lzu.edu.cn(C.Sun); changqq12@lzu.edu.cn(Q.Chang).\\
}
\vspace{-1cm}

\section{Introduction}
\noindent

In this paper, we consider the following damped degenerate wave equation
\begin{equation} \label{mp00}
\begin{cases}
\partial_{tt}u(x,t) + \beta(t) \partial_tu(x,t) = \mathcal{L}u(x,t) + f(u(x,t)) &\text{in }\Omega,~t>0, \\
u(x,t) = 0 & \text{in }\partial\Omega,~t\geq0, \\
u(x,0) = u_0(x),~\partial_tu(x,0)=u_1(x) & \text{in }\Omega ,
\end{cases}
\end{equation}
where $\Omega\subset\mathbb{R}^N$ is a bounded smooth domain, $(u_0,u_1)$ is the initial data and the nonlinear term $f$  is a locally Lipschitz continuous function with critical growth condition and dissipativity condition.

The degenerate elliptic operator $\mathcal{L}$ is  defined as the following
\begin{equation*}
  \mathcal{L}u := \sum_{i,j=1}^N\partial_{x_i}(a_{ij}\partial_{x_j}u)
\end{equation*}
and the functions $\{a_{ij}\}$ are measurable in $\mathbb{R}^N$ and $a_{ij}=a_{ji}$, that is X-elliptic with respect to the family of vector fields $X=\{X_1,\cdots,X_m\}$. G. Folland constructed a general theory of ``subelliptic" regularity on a class of Lie groups which is sufficiently broad to admit a wide variety of applications to more general problems, namely the class of ``stratified groups" in \cite{Folland1975}, in which the subelliptic operator is proved playing the same role in Lie group as the standard Laplace operator in Euclidean space. In 1996, N. Garofalo and D. Neieu figured out isoperimetric and Sobolev inequalities of general families of vector fields  that include the H$\ddot{o}$rmander type as a special case in \cite{Garofalo1996}. Based on these papers, the notion of X-elliptic operator, which will be recalled in the Subsection 2.1, was first introduced in 2000 in the paper \cite{fin}.  In paper \cite{sz1} the authors defined a metric associated to the X-elliptic operators that plays the same role as the Euclidian metric for the standard Laplacian.  Using this metric they obtained Sobolev type embedding
theorems for such operators in \cite{BEE} and \cite{BE}, and the Sobolev embeddings allow us to solve the semilinear problem locally under sub-critical growth restrictions on the non-linearity. Meanwhile, the growth restrictions are determined by the homogenous dimension $Q$ with respect to the metric associated to the X-elliptic operators.
More recently, X-elliptic operators were studied intensely in \cite{CE,HAr,LT} (see more references therein), where a maximum principle, a non-homogenous Harnack inequality, Hardy type inequalities and Liouville theorem were obtained.

The wave equation equation formulated as
\begin{equation*}
\partial_{tt}u(x,t) + \beta  \partial_tu(x,t) = \triangle u(x,t) + f(u(x,t)) \qquad\text{in }\Omega,~t>0,
\end{equation*}
is an important second-order
nonlinear partial differential equation for the description of waves as
they occur in classical physics such as sound waves, light waves and water waves. It arises in fields like acoustics, electromagnetics, and fluid dynamics (see in \cite{MPM}). The wave equation alone does not specify a physical phenomenon, we usually set further conditions to describe some specific phenomenon, such as initial condition, boundary conditions and so on. Besides the general conditions for the wave equation, one can set additional conditions for the damping coefficient due to different viscous medium where the motion taken place. In \cite{POD}, the authors assume the damping coefficient asymptotically periodic depends on the Euclidean space and they  proved local and global energy decay for the wave equation. And in \cite{bto} and \cite{betate}, the coefficient $\beta(t):\mathbb{R}\to\mathbb{R}$ is considered bounded from below and above by some positive real numbers $m$ and $M$ respectively. They have shown that this equation has a pullback attractor and the pullback attractor is gradient-like, i.e. given as the union of the unstable manifolds of a finite set of hyperbolic global trajectories.

In this paper, we investigate the case that the damping coefficient $\beta(t):~\mathbb{R}\to\mathbb{R}$  is a continuous function belongs to $L_{loc}^{\infty}(\mathbb{R})$ and for some constant $C_{\beta}\geq1$ satisfies the following assumptions:
\begin{equation}\label{bt0}
    1\leq\int_t^{t+1}\beta(s)\mathrm{d}s = \int_t^{t+1}\alpha(s)\mathrm{d}s - \int_t^{t+1}\gamma(s)\mathrm{d}s \leq C_{\beta}~\text{for any}~t\in \mathbb{R}
\end{equation}
and
\begin{equation}\label{bt1}
    \int_{-\infty}^{+\infty}\gamma(s)\mathrm{d}s< \infty
\end{equation}
where functions $\alpha(t)$ and $\gamma(t)$ are the positive and negative parts of the function $\beta(t)$ respectively,  i.e.
$$\alpha(t):=\max\{\beta(t),0\}\qquad \text{and}\qquad \gamma(t):=\max\{-\beta(t),0\}.$$
Note that the constant $1$ in the left side of \eqref{bt0} can be replaced with any positive constant $C_{\beta_0}$ and as long as $C_{\beta}\geq C_{\beta_0}$ the demonstration in this paper is still valid.

One can easily find examples for $\beta(t)$ satisfy the above assumptions
   \begin{enumerate}
     \item $\beta(t)\equiv c>0$ (e.g., the case considered in \cite{II1,cV,jk}) or $0<m\leqslant \beta(t)\leqslant M$ (e.g., the case considered in \cite{bto,betate}) clearly satisfy the assumptions;
     \item $\beta(t)=\frac{3}{2}\sin(2\pi t)+\frac{|t|^2}{2|t|^2+1}+1.$
   \end{enumerate}

The main difficulty in this paper is that it is not assumed that $\beta(t)$ is a bounded continuous function and $\beta(t)$ has fixed signs. For instance, there is no big difference between the case in \cite{bto,betate}  and case that the coefficient $\beta(t)$ be considered as a positive constant technically, i.e. under suitable assumptions on $f$, the process associated with the problem \eqref{mp00} must be dissipative if the term $\beta(t)$ stays positive. However with the assumptions \eqref{bt0}-\eqref{bt1}, the term $\beta(t)$ may be negative for some subset of the real line. The negative part of the damping term is allowed to vanish but the positive part will not, so that the damping is effective at infinity.
 Our hypothesis for the coefficient $\beta(t)$, e.g. \eqref{bt1}, may not be the optimal one, but to our best knowledge the case is the first to be considered in this paper, especially by the so-called attractor. Furthermore, the problem \eqref{mp00} with assumptions \eqref{bt0}-\eqref{bt1} has some mathematical obstacles to overcome, for example, it is not known whether one can remove the technical condition \eqref{bt1} (we believe that \eqref{bt1} is unnecessary).



The aim of this paper is to establish the existence of a pullback attractor to the problem \eqref{mp00} under the assumptions that $f$ satisfies growth and dissipativity conditions and the time-dependent term $\beta(t)$ satisfies conditions \eqref{bt0}{\color{blue}-}\eqref{bt1}.
It is well-known that the long-time behavior of dynamical systems generated by evolution equations of mathematical physics can be described in terms of attractors of corresponding semigroups in many situations. Nonlinear wave phenomena occurs in various systems in physics, engineering, biology and geosciences\cite{jk,cV,jm,bto}. At the macroscopic level, wave phenomena may be modelled by hyperbolic wave-type partial differential equations.
As to the degenerate partial differential equations, plenty of papers are devoted to the study of long time behavior of solutions to degenerate parabolic and hyperbolic equations, such as the semilinear heat equation and the semilinear damped wave equation involving the degenerate operator X-elliptic instead of the classical Laplace operator (see \cite{AMX,AP,LS} and references therein),  where the existence of the global attractor and finite fractal dimension were obtained.


The non-autonomous systems are of great importance and interest as they appear in many applications to natural science. For a non-autonomous dynamical system such as \eqref{mp00}, the solution does not define a semigroup and, instead, it defines a two-parameter process {or} cocycle. Hence it is reasonable to describe asymptotic dynamics for the problem $\eqref{mp00}$ applying the theory of pullback attractors. We will recall the basic concepts for non-autonomous dynamical systems in Subsection 2.2. The strongly damped wave equation with time-dependent terms we considered in this paper was studied by several authors in \cite{bto} and \cite{betate} (see references therein). With the assumption that $\beta(t)$ is bounded by some positive numbers from above and below they achieved the existence, regularity, continuity, characterization and continuity of characterization of pullback attractors for the problem in these papers. Moreover, they proved that the evolution process associated with the strongly damped wave equation with time-dependent terms is a gradient-like evolution (homoclinic structures do not exist) in both \cite{bto} and \cite{betate}.

The content of the paper is as follows. In Section 2 we recall the concept of X-elliptic operator and some general results of pullback attractors. Then we obtain the local and global well-posedness of the Cauchy problem for \eqref{mp00} and the existence of a family of absorbing sets in Section 3. Finally the existence of the pullback attractor is proved in Section 4 by using of the results based on asymptotic compact processes.

%

\section{Preliminaries}
\noindent

In this section we recall the concept of X-elliptic operator and some general results of pullback attractors.

\subsection{Hypotheses and main results for $X$-elliptic operator}
\noindent

The partial differential operators we are dealing with are of the following type (see \cite{AMX})
\begin{equation}
\mathcal{L}u:=div(\mathbb{A}\nabla u ) =  \sum_{i,j=1}^N\partial_{x_i}(a_{ij}\partial_{x_j}u),
\end{equation}
where the matrix $\mathbb{A}$ is supposed to be symmetric and positive, the functions $a_{ij}$ are measurable in $\mathbb{R}^N$ and $a_{ij}=a_{ji}$. We assume that there exists a family $X:=\{X_1,\cdots,X_m\}$ of vector fields in $\mathbb{R}^N$, $X_j=(\alpha_{j1},\cdots,\alpha_{jN})$, $j=1,\cdots,m$, such that the functions $\alpha_{ij}$ are locally Lipschitz continuous in $\mathbb{R}^N$. As usual, we identify the vector-valued function $X_j$ with the linear first order partial differential operator
\begin{equation*}
  X_j = \sum_{k=1}^N\alpha_{jk}\partial_{x_k},\qquad j=1,\cdots,m.
\end{equation*}

We suppose, moreover, we assume the following hypotheses are satisfied
\begin{itemize}
  \item[(H1)] The control distance $d=d_X$ is well-defined, the metric spaces $(\mathbb{R}^N,d)$ is complete and the d-topology is the Euclidean one. Moreover there exists $A_0>1$ such that the following doubling condition holds
      \begin{equation}\label{pr1}
        0 < |B_{2r}| \leq A_0|B_r|,
      \end{equation}
      for every d-ball $B_r$ of radius $r$, $0<r<\infty$, hereafter $|E|$ will denote the Lebesgue measure of the set $E\subset\mathbb{R}^N$.
  \item[(H2)] There exist positive constants $C,\nu\geq1$ such that the following Poincar\'{e} inequality holds
      \begin{equation*}
        \fint_{B_r}|u-u_r|\mathrm{d}x \leq Cr\fint_{B_{\nu r}}|Xu|\mathrm{d}x \qquad \forall~u\in C^1(\bar{B_{\nu r}}),
      \end{equation*}
      where the mean value $u_r$ is $u_r = \fint_{B_r}u := \frac{1}{|B_r|}\int_{B_r}u$ and $Xu$ denotes the X-gradient of $u$, i.e.
      \begin{equation*}
        Xu = (X_1u,\cdots,X_mu).
      \end{equation*}
\end{itemize}

\begin{definition}The operator $\mathcal{L}$ is called uniformly X-elliptic if there exists a constant $C>0$ such that
\begin{equation*}
  \frac{1}{C}\sum_{j=1}^m\langle X_j(x),\xi\rangle^2 \leq \sum_{i,j=1}^Na_{ij}(x)\xi_i\xi_j \leq C\sum_{j=1}^m\langle X_j(x),\xi\rangle^2 \qquad x,\xi\in\mathbb{R}^N,
\end{equation*}
where $\langle\cdot,\cdot\rangle$ denotes the usual inner product in $\mathbb{R}^N$.
\end{definition}

We define the Hilbert space $H$ as the closure of $C_0^1(\Omega)$ with respect to the form
\begin{equation*}
  \|u\|_H := \big(\sum_{j=1}^m\|X_ju\|^2_{L^2(\Omega)}\big)^{\frac{1}{2}},\qquad u\in C_0^1(\Omega),
\end{equation*}
and assume the following Sobolev-type embedding property holds(see \cite{LS} for details):
\begin{equation}\label{qr1}
  H\hookrightarrow L^p(\Omega)\qquad \forall~p\in\Big[1,\frac{2Q}{Q-2}\Big],
\end{equation}
where $Q\leq N$ is the corresponding dimension for the operator $\mathcal{L}$ and the embedding is compact for every $p\in[1,\frac{2Q}{Q-2})$.

The operator $A:= -\mathcal{L}$ is positive, self-adjoint in $L^2(\Omega)$, and has compact inverse. Consequently, there exists an orthonormal basis of $L^2(\Omega)$ of eigenfunctions $\{\psi_j\}\in H,~j\in \mathbb{N}$ of $A$ with eigenvalues
\begin{equation*}
  0<\mu_1<\mu_2\leq\mu_3\leq \cdots, \quad\text{and}~\mu_j\to\infty ~\text{as}~j\to\infty.
\end{equation*}
Follows the theory that constructed the fractional power spaces associated with $A$ in \cite{AMX}, we denote for every $\alpha > 0$
\begin{equation*}
  X^{\alpha} = \left(\mathcal{D}(A^{\alpha}),\langle\cdot,\cdot\rangle_{X^{\alpha}}\right)=\left\{\psi = \sum_{j\in\mathbb{N}}c_j\psi_j,~c_j\in\mathbb{R}\big|\sum_{j\in\mathbb{N}}\mu_j^{2\alpha}c_j^2 < \infty\right\} ,
  \end{equation*}
where the inner product in $X^{\alpha}$ is given by $\langle u, v\rangle_{X^{\alpha}} = \langle A^{\alpha}u, A^{\alpha}v\rangle_{X^0}$ for any $u,v\in\mathcal{D}(A^{\alpha})$ and
\begin{equation*}
  A^{\alpha}\psi = A^{\alpha}\sum_{j\in\mathbb{N}}c_j\psi_j = \sum_{j\in\mathbb{N}}\mu_j^{\alpha}c_j\psi_j.
\end{equation*}

With this notation, we see that $X^{\frac{1}{2}} = H$, $X^0 = L^2(\Omega)$.
Using complex interpolation (as in Lions-Magenes) we conclude that for $0<\delta <1$ the embedding
\begin{equation}\label{q1}
  \left[X^0,X^{\frac{1}{2}}\right]_{\delta} = X^{\frac{\delta}{2}} \hookrightarrow \left[L^2(\Omega),L^q(\Omega)\right]_{\delta} = L^p(\Omega),
\end{equation}
is linear and bounded where $q=\frac{2Q}{Q-2}$ and $\frac{1}{p} = \frac{1}{2} - \frac{\delta}{Q}$.

Meanwhile by duality:
\begin{equation}\label{q2}
  L^2(\Omega)\subset L^{p^{\prime}}(\Omega) \subset X^{-\frac{\delta}{2}} \qquad \text{for} \quad \frac{1}{p} + \frac{1}{p^{\prime}} = 1\quad \text{and} \quad \frac{1}{p} = \frac{1}{2} - \frac{\delta}{Q},
\end{equation}
in particular we have
\begin{equation}\label{q3}
  X^{\frac{1-\delta}{2}}\subset L^p(\Omega) \qquad \frac{1}{p} = \frac{1}{2} + \frac{1-\delta}{Q}.
\end{equation}


We now state the local existence theorem for
\begin{equation*}
  \partial_tu = Au + f(u).
\end{equation*}
Assume that the operator $A$ generates an analytic semigroup in $X^0 := L^2(\Omega)$ and the nonlinear term $f: \mathbb{R} \times X^1 \to X^{\alpha}$ is such that, for each $r > 0$, there are constants $ C = C(r)$ such that, if $\|x\|_{X^1}$, $\|y\|_{X^1} \leq r$, then
\begin{equation}\label{con1}
  \|f(t,x)-f(t,y)\|_{X^{\alpha}} \leq C(r)h(t)\|x-y\|_{X^1},
\end{equation}
where $h(t)\equiv 1$ if $\alpha = 1$, and if $\alpha < 1$, then one can take $h(t) = \psi(t - t_0)|t - t_0|^{-\alpha}$ with $\psi: [0,\infty) \to [0,\infty)$ an increasing continuous function such that $\psi(0) = 0$.

\begin{theorem}[\cite{AIN}]\label{et}Suppose either that
\begin{itemize}
  \item[(i)] $A$ is sectorial and $\alpha\in (0,1]$, or
  \item[(ii)] -$A$ generates a strongly continuous semigroup and $\alpha = 1$.
\end{itemize}
If $f$ satisfies \eqref{con1}, then for any $r>0$, there exists a time $\tau_0 > 0$ such that, for any $x_0\in B_{X^1}(0,r)$, there is a unique mild solution $x(\cdot;t_0,x_0) : [t_0, t_0 + \tau_0] \to X^1$ of
\begin{equation*}
  \Dot{x} = -Ax + f(t,x) \qquad x(t_0) = x_0.
\end{equation*}
These solutions depends continuously on the initial data: if $x_0,~z_0\in B_{X^1}(0,r)$, then
\begin{equation*}
  \|x(t;t_0,x_0) - x(t;t_0,z_0)\|_{X^1} \leq C\|x_0 - z_0\|_{X^1} \qquad \text{for all} \qquad t\in[t_0, t_0 + \tau_0].
\end{equation*}
When $A$ is sectorial and $x_0,~z_0\in B_{X^1}(0,r)$, we have $x\in C((t_0, t_0 + \tau_0], X^{1+\theta})$ for all $0 \leq \theta < \alpha$,
\begin{equation*}
  (t-t_0)^{\theta} \|x(t;t_0,x_0)\|_{X^{1+\theta}} \to 0 \quad \text{as} \quad t \to t_0, \quad 0 < \theta < \alpha,
\end{equation*}
and for $0 \leq \theta \leq \theta_0 < \alpha$,
\begin{equation*}
  (t-t_0)^{\theta}\|x(t;t_0,x_0)-x(t;t_0,z_0)\|_{X^{1+\theta}} \leq C_{\theta_0}\|x_0 - z_0\|_{X^1} \quad \text{for all} \quad t\in[t_0, t_0 + \tau_0].
\end{equation*}
\end{theorem}

We also state a nonstandard extension of the classic Gronwall's inequality(see \cite{IC}).

\begin{lemma}[\cite{IC}]\label{G1}Assume that $f(t)$ is a continuous function on the interval $[a,b]$ such that
\begin{equation}\label{g3}
  \lim_{\delta\to 0}\inf_{\delta<0}\frac{1}{|\delta|}[f(t+\delta)-f(t)] \geq -m(t)
\end{equation}
for almost all $t\in(a,b)$, where $m(t)\in L^1(a,b)$. Then
\begin{equation}\label{g4}
  f(t_1)-f(t_2) \leq \int_{t_1}^{t_2}m(\tau)\mathbb{d}\tau \qquad \text{for all}\qquad a\leq t_1 < t_2 \leq b.
\end{equation}
\end{lemma}

\begin{lemma}\label{G2}Let $\psi(t)$ and $g(t)$ be given scalar functions from $L^1_{loc}(\mathbb{R})$. Assume that a continuous function $h(t)$ defined on $\mathbb{R}$ satisfies the inequality
\begin{equation}\label{g1}
  h(t) + \int_s^t\psi(\tau)h(\tau)\mathrm{d}\tau \leq h(s) + \int_s^tg(\tau)\mathrm{d}\tau
\end{equation}
for all $t\geq s$. Then
\begin{equation}\label{g2}
  h(t) \leq h(s)\exp\left\{-\int_s^t\psi(\sigma)\mathrm{d}\sigma\right\} + \int_s^tg(\tau)\exp\left\{-\int_\tau^t\psi(\sigma)\mathrm{d}\sigma\right\}\mathrm{d}\tau
\end{equation}
for all $t\geq s$.
\end{lemma}

\begin{proof}The original Lemma in \cite{IC} with condition $\psi(t)$ and $g(t)$ be given scalar functions from $L^1_{loc}(\mathbb{R^+})$. The idea of the proof presented here is borrowed from \cite{IC}.
We apply Lemma \ref{G1} to the function
\begin{equation*}
  f(t) = h(t)exp\left\{\int_s^t\psi(\sigma)\mathrm{d}\sigma\right\}.
\end{equation*}
It follows from \eqref{g1} with $s = t + \delta$, $\delta<0$, that
\begin{equation*}
  \begin{aligned}
     &f(t+\delta) - f(t) \\
     =& h(t+\delta)\exp\left\{\int_s^{t+\delta}\psi(\sigma)\mathrm{d}\sigma\right\} - h(t)\exp\left\{\int_s^t\psi(\sigma)\mathrm{d}\sigma\right\}\\
     =& [h(t+\delta)-h(t)]\exp\left\{\int_s^{t+\delta}\psi(\sigma)\mathrm{d}\sigma\right\} - h(t)\left[\exp\left\{\int_s^{t+\delta}\psi(\sigma)\mathrm{d}\sigma\right\}-\exp\left\{\int_s^t\psi(\sigma)\mathrm{d}\sigma\right\}\right] \\
     \geq& \left[\int_{t+\delta}^t\psi(\tau)h(\tau)\mathrm{d}\tau - \int_{t+\delta}^tg(\tau)\mathrm{d}\tau\right]\exp\left\{\int_s^{t+\delta}\psi(\sigma)\mathrm{d}\sigma\right\}\\
      &+ h(t)\left[\exp\left\{\int_s^{t+\delta}\psi(\sigma)\mathrm{d}\sigma\right\}-\exp\left\{\int_s^t\psi(\sigma)\mathrm{d}\sigma\right\}\right].
  \end{aligned}
\end{equation*}
Take the lower limit in both sides of the inequality above as $|\delta|\to 0$, we can obtain \eqref{g3} with $m(t) = g(t)\exp\{\int_s^t\psi(\sigma)\mathrm{d}\sigma\}$. Therefor the application of Lemma \ref{G1} yields the inequality in \eqref{g2}.
\end{proof}

\subsection{General theory of infinite dimensional dynamical system}
\noindent

We begin with precise definitions of the notions of a process and the existence theory for pullback attractors. The dynamics occurs in a phase space $X$, which represents all possible states of the underlying systems. In the general treatment here we will take the phase space to be a metric space $(X,d)$, although in practical applications the space $X$ will be specialise to Banach, Hilbert, or Euclidean spaces as required. We will denotes by $\mathcal{C}(X)$ the set of all continuous transformations from $X$ into itself. For further details we refer to \cite{AF} and \cite{IC}.

A process in $X$ is a family of maps $\{S(t,s):t\geq s\}$ in $\mathcal{C}(X)$ such that properties
\begin{itemize}
  \item [(1)]$S(t,t) = I$, for all $t\in\mathbb{R}$,
  \item [(2)]$S(t,s) = S(t,\tau)S(\tau,s)$, for all $t\geq\tau\geq s$,
  \item [(3)]$(t,s,x)\to S(t,s)x$ is continuous, $t\geq s$, $x\in X$.
\end{itemize}

Throughout this section, $S(\cdot,\cdot)$ is a process on a metric space $(X,d)$.

\begin{definition}[\cite{AIN}](\textbf{Strongly bounded dissipative})\label{sbd} We say that a process $S(\cdot,\cdot)$ is strongly bounded dissipative if for each $t\in\mathbb{R}$ there is a bounded subset $B(t)$ of $X$ that pullback attracts bounded subsets of $X$ at time $\tau$; that is, given a bounded subset $D$ of $X$ and $\tau\leq t$, $\lim_{s\to-\infty}dist(S(\tau,s)D,B(t))=0$.
\end{definition}

\begin{definition}[\cite{AIN}](\textbf{Pullback asymptotically compact}) A process $S(\cdot,\cdot)$ in a metric space $X$ is said to be pullback asymptotically compact if, for each $t\in\mathbb{R}$, each sequence $\{s_k\}\leq t$ with $s_k\to-\infty$ as $k\to\infty$, and each bounded sequence $\{x_k\}\in X$ the sequence $\{S(t,s_k)x_k\}$ has a convergent subsequence.
\end{definition}

\begin{definition} The pullback $\omega$-limit set at time $t$ of a subset $B$ of $X$ is defined by
\begin{equation*}
  \omega(B,t):=\bigcap_{\sigma\leq t}\overline{\bigcup_{s\leq\sigma}S(t,s)B},
\end{equation*}
or, equivalently,
\begin{eqnarray*}
  \omega(B,t)= \big\{y\in X: \text{there are sequences}~\{s_k\}\leq t,~s_k\to-\infty~\text{as}~k\to\infty, \\
  \text{and}~\{x_k\}\in B,~\text{such that}~y=\lim_{k\to\infty}S(t,s_k)x_k\big\}.
\end{eqnarray*}
\end{definition}

The following theorem gives a sufficient condition for the existence of a compact pullback attractor $\mathcal{A}(\cdot)$ that is bounded in the past, i.e.
\begin{equation*}
  \bigcup_{s\leq t}\mathcal{A}(s)
\end{equation*}
is bounded for each $t\in\mathbb{R}$.

\begin{theorem}[\cite{AIN}]\label{eth}If a process $S(\cdot,\cdot)$ is strongly pullback bounded dissipative and pullback asymptotically compact and $B(\cdot)$ is a family of bounded subsets of $X$ such that, for each $t\in\mathbb{R}$, $B(\cdot)$ pullback attracts bounded subsets of $X$ at time $\tau$ for each $\tau\leq t$, then $S(\cdot,\cdot)$ has a compact pullback attractor $\mathcal{A}(\cdot)$ such that $\mathcal{A}(t)=\omega(\bar{B}(t),t)$ and $\bigcup_{s\leq t}\mathcal{A}(s)$ is bounded for each $t\in\mathbb{R}$.
\end{theorem}


\section{Global existence and dissipativity of solutions}
\noindent

We will consider the following degenerate damped hyperbolic initial value problem on a smooth bounded domain $\Omega\subset\mathbb{R}^N$:
\begin{equation} \label{mp}
\begin{cases}
\partial_{tt}u(x,t) + \beta(t) \partial_tu(x,t) = \mathcal{L}u(x,t) + f(u(x,t)) &x\in\Omega,~t>0, \\
u(x,t) = 0 & x\in\partial\Omega,~t\geq0, \\
u(x,0) = u_0(x),~\partial_tu(x,0)=u_1(x) & x\in\Omega ,
\end{cases}
\end{equation}
where the initial data $(u_0,u_1)\in V$, and $V := H\times L^2(\Omega)$ is the natural phase space for the problem equipped with the norm
\begin{equation*}
  \|w\|_V := \left(a(u,u) + \|v\|^2_{L^2(\Omega)}\right)^{\frac{1}{2}},\qquad w = (u,v)\in V,
\end{equation*}
where we define the bilinear form $a(\cdot,\cdot)$ as $a(u,v) = (-\mathcal{L}u,v)_{L^2(\Omega)}$.

%

Assume that the function $\beta(t)\in L_{loc}^{\infty}(\mathbb{R})$, $\beta(t):\mathbb{R}\to\mathbb{R}$ is a continuous function which satisfies conditions \eqref{bt0}-\eqref{bt1} we proposed in Sect.1 and the nonlinear term $f$ is locally Lipschitz continuous with critical growth condition
\begin{equation}\label{fcon}
  |f(u)-f(v)| \leq c|u-v|(1 + |u|^{\gamma} + |v|^{\gamma}), \quad u,v\in\mathbb{R},
\end{equation}
for some constant $c\geq0$, where we assume $0<\gamma \leq \frac{q-2}{2}$ and $q=\frac{2Q}{Q-2}$.

We also need the following sign condition to ensure the global existence of solutions and allow to characterize their dissipation,
\begin{equation}\label{scon}
  \limsup_{|u|\to\infty}\frac{f(u)}{u} < \mu_1, \qquad \forall~u\in\mathbb{R},
\end{equation}
where $\mu_1>0$ denotes the first eigenvalue of the operator $-\mathcal{L}$ on $\Omega$ with homogeneous Dirichlet boundary conditions. Then, by spliting the domain in $B_R\subset\mathbb{R}$ and its complementary space we can deduce that there exist constants $0\leq c_0 < \mu_1$ and $c_1\in\mathbb{R}$ such that
\begin{equation}\label{sconc}
  uf(u) \leq c_1|u| + c_0|u|^2~\text{for all}~ u\in\mathbb{R}.
\end{equation}
Moreover, according to the argument that deduced \eqref{sconc} above, the constant $c_0$ can be determined in $(0,~\frac{\mu_1}{2})$ if we choose $B_R$ properly.

%

\begin{definition}We call $u(t,s)$ a local weak solution of \eqref{mp} if for any $s\in\mathbb{R}$ there exists $T>0$ such that
\begin{eqnarray*}
  \begin{aligned}
     u\in C([s,T);X^{\frac{1}{2}})\cap C^1((s,T);L^2(\Omega)), \\
   u(s)=u_0,~\partial_tu(s)=u_1, \qquad\qquad
  \end{aligned}
\end{eqnarray*}
where $(u_0,~u_1)\in V$ and $u(t,s)$ satisfies the equation
\begin{equation*}
  \frac{d^2}{dt^2}\langle u(t), v\rangle_{L^2(\Omega)} + \beta(t)\frac{d}{dt}\langle u(t),v\rangle_{L^2(\Omega)} = a(u(t),v) + \langle f(u),v\rangle \quad \forall~v\in X^{\frac{1}{2}},~t\in(s,T).
\end{equation*}
\end{definition}
\subsection{Global well-posedness}
\indent

We formulate problem \eqref{mp} in the abstract form:
\begin{eqnarray}\label{rmp}
  \partial_tw^{T} &=& \hat{A}w^{T} + \hat{f}(w),\qquad w^{T} = \left(\begin{array}{c}
  u\\
  v
  \end{array}
  \right) \nonumber\\
  w_{t=0} &=& w_0 = (u_0,v_0)^T,
\end{eqnarray}
where $v:=\partial_tu$ and $w:=(u,v)^T$. Moreover,
$\hat{A}$ and $\hat{f}$ are defined as
\begin{equation*}
  \hat{A} :=\left( \begin{array}{cc}
               0 & Id \\
               -A & 0
             \end{array}\right),\qquad
  \hat{f}(w) := \left(\begin{array}{c}
                  0 \\
                 -\beta(t)v+ f(u)
                \end{array}\right),
\end{equation*}
where $A$ denotes the operator $-\mathcal{L}$ in $L^2(\Omega)$ with homogenous Dirichlet boundary conditions.

Due to Theorem \ref{et}, local well-posedness is a consequence of the following result, which shows that the nonlinearity term is locally Lipschitz from $V$ to $V$.

\begin{lemma}\label{let} Let the assumption \eqref{fcon} for nonlinearity term $f$ holds, then $\hat{f}:V\to V$ is locally Lipschitz continuous.
\end{lemma}
\begin{proof}For any bounded domain $D\subset V$, for all $w = (u,v)^T,~\bar{w}=(\bar{u},\bar{v})^T\in D$, using the growth condition \eqref{fcon} and H$\ddot{o}$lder's inequality with $p=\frac{Q}{Q-2}$ and $p^{\prime}=\frac{Q}{2}$ we have:
\begin{eqnarray}\label{ly1}
  \|f(u)-f(\bar{u})\|^2_{L^2(\Omega)} &\leq& c\int_{\Omega}|u-\bar{u}|^2\big(1 + |u|^{2\gamma}+|\bar{u}|^{2\gamma}\big) \nonumber\\
   &\leq& c\|u-\bar{u}\|^2_{L^{2p}(\Omega)}\left(1 + \|u\|^{2\gamma}_{L^{2\gamma p^{\prime}}(\Omega)} + \|\bar{u}\|^{2\gamma}_{L^{2\gamma p^{\prime}}(\Omega)}\right)\nonumber \\
   &\leq& c\|u-\bar{u}\|^2_{L^{q}(\Omega)}\left(1 + \|u\|^{2\gamma}_{L^q(\Omega)} + \|\bar{u}\|^{2\gamma}_{L^q(\Omega)}\right)\nonumber \\
   &\leq& c\|u-\bar{u}\|^2_{X^{\frac{1}{2}}},
\end{eqnarray}
where we applied the embedding $X^{\frac{1}{2}}\hookrightarrow L^r(\Omega)$ with $0<r\leq q= \frac{2Q}{Q-2}$ in the last step, the constant $c$ depends on the domain $D$ and may vary in lines.

Let now prove the lemma by applying the equality \eqref{ly1}:
\begin{eqnarray*}
  \|\hat{f}(w)-\hat{f}(\bar{w})\|_V &=& \|-\beta(t)(v-\bar{v}) + f(u) - f(\bar{u})\|_{L^2(\Omega)} \\
       &\leq& c\big(\|v-\bar{v}\|_{L^2(\Omega)} + \|f(u) - f(\bar{u})\|_{L^2(\Omega)}\big) \\
   &\leq& c\big(\|v-\bar{v}\|_{L^2(\Omega)} + \|u-\bar{u}\|_{X^{\frac{1}{2}}}\big) \\
   &\leq& c \|w-\bar{w}\|_V,
\end{eqnarray*}
where the constant $c$ depends on the domain $D$ and may vary in lines.
\end{proof}

\begin{theorem}\label{etg}We assume the function $f$ satisfies the growth restriction \eqref{fcon} and $\beta(t)$ satisfies condition \eqref{bt0}-\eqref{bt1}. Then for every initial data $w_0=(u_0,v_0)^T\in V$ there exists a unique local weak solution of problem \eqref{mp} defined on the maximal interval of existence $[s,T)$ and
\begin{equation*}
  w\in C([s,T);V)\cap C^1((s,T);V).
\end{equation*}
The solution satisfies the variation of constant formula
\begin{equation*}
  w(t) = e^{\hat{A}(t-s)}w_0 + \int_s^te^{\hat{A}(t-\tau)}\hat{f}(w(s))\mathrm{d}\tau,
\end{equation*}
and either $T=\infty$ or, if $T<\infty$ then
\begin{equation*}
  \limsup_{t\to T} \|w(t)\|_V =\infty.
\end{equation*}
\end{theorem}

To show the global existence of solutions, multiplying by $\partial_tu$ in both sides of  the equation \eqref{mp} we obtain:
\begin{equation}\label{ut1}
  \frac{1}{2}\frac{d}{dt}\|\partial_tu\|^2_{L^2(\Omega)} + \beta(t)\|\partial_tu\|^2_{L^2(\Omega)} + \frac{1}{2}\frac{d}{dt}a(u,u) - \frac{1}{2}\frac{d}{dt}\int_{\Omega}F(u)\mathrm{d}x = 0,
\end{equation}
where $F(u)=\int_0^u f(s)\mathrm{d}s$ denotes the primitive of $f$. Since the Lemma \ref{let} is valid, Theorem \ref{et} implies that if $u$ is a mild solution of \eqref{mp} then $(u,\partial_tu)\in C([s,T];V)$.

Then integrating \eqref{ut1} over $[s,t]$ we have:
\begin{equation}\label{ut2}
  \begin{aligned}
     &\frac{1}{2}\|\partial_tu\|^2_{L^2(\Omega)} + \int_s^t \beta(\tau)\|\partial_tu\|^2_{L^2(\Omega)}\mathrm{d}\tau + \frac{1}{2}a(u,u) - \int_{\Omega}F(u)\mathrm{d}x \\
     = &\frac{1}{2}\|\partial_tu(s)\|^2_{L^2(\Omega)} + \frac{1}{2}a(u(s),u(s)) - \int_{\Omega}F(u_0)\mathrm{d}x
  \end{aligned}
\end{equation}
Using the growth restriction \eqref{fcon} we obtain
\begin{equation}\label{ut3}
  \int_{\Omega}|F(u_0)|\mathrm{d}x \leq C\|u_0\|^{\gamma+2}_{L^{\gamma+2}(\Omega)},
\end{equation}
where $\gamma+2 = \frac{q-2}{2} + 2 < q$, hence the embedding $L^{\gamma +2}(\Omega)\hookrightarrow X^{\frac{1}{2}}$ holds.

On the other hand, the sign condition \eqref{scon} and Young's inequality yield the estimate:
\begin{eqnarray}\label{ut4}
  \frac{1}{2}\|u(t)\|^2_{X^{\frac{1}{2}}} - \int_{\Omega}F(u)\mathrm{d}x &\geq& \frac{1}{2}\|u\|^2_{X^{\frac{1}{2}}} -\frac{1}{2}c_0\|u\|^2_{L^2(\Omega)} - \int_{\Omega}c_1|u|\mathrm{d}x \nonumber\\
   &\geq& \frac{1}{2}\|u\|^2_{X^{\frac{1}{2}}} - \frac{1}{2}(c_0+\epsilon)\|u\|^2_{L^2(\Omega)} - C_{\epsilon} \nonumber\\
   &\geq& \frac{1}{2} \|u\|^2_{X^{\frac{1}{2}}}\big(1-\frac{c_0+\epsilon}{\mu_1}\big) - C_{\epsilon} ,
\end{eqnarray}
for small $\epsilon>0$ and some constant $C_{\epsilon}\geq0$, where we used Poincar\'{e}'s inequality in the last step.

Set
\begin{equation}
M_0 = C\left(\frac{1}{2}\|\partial_tu(s)\|^2_{L^2(\Omega)} + \frac{1}{2}\|u(s)\|^2_{X^{\frac{1}{2}}} + \|u(s)\|^{\gamma+2}_{L^{\gamma+2}(\Omega)}\right) + C_{\epsilon},
 \end{equation}
and plug the inequality \eqref{ut4} into the inequality \eqref{ut2} we obtain:
\begin{equation}\label{ut5}
  \frac{1}{2}\|\partial_tu\|^2_{L^2(\Omega)} \leq M_0 + \int_s^t\gamma(\tau)\|\partial_tu\|^2_{L^2(\Omega)}\mathrm{d}\tau,
\end{equation}
the Gronwall Lemma (see in \cite{AIN}) and assumption in \eqref{bt1} imply that:
\begin{eqnarray}\label{utc}
  \|\partial_tu\|^2_{L^2(\Omega)} &\leq& 2M_0\exp\left\{2\int_s^t\gamma(\tau)\mathrm{d}\tau\right\}\nonumber\\
   &\leq&  2M_0 e^{2r_0},
\end{eqnarray}
where $r_0 := \int_0^t\gamma(s)\mathrm{d}s $.

Combine \eqref{ut4} and \eqref{utc},  we can also obtain the following inequalities:
\begin{eqnarray*}
  \frac{1}{2} \|u\|^2_{X^{\frac{1}{2}}}\left(1-\frac{c_0+\epsilon}{\mu_1}\right) &\leq& M_0 + \int_s^t\gamma(\tau)\|\partial_tu\|^2_{L^2(\Omega)}\mathrm{d}\tau \\
   &\leq& M_0 + 2r_0M_0e^{2r_0}.
\end{eqnarray*}

We can see from the above estimates that the solution $u$ is uniformly bounded in time $t\in[s,T)$ in $V= X^{\frac{1}{2}}\times L^2(\Omega)$, and therefore exists globally, i.e., $T=\infty$.


\begin{remark}\label{rmk1}From the discussions above we can see that the assumption \eqref{bt1} plays an important role to prove the global existence. In fact, the positive part $\alpha(t)$ of coefficient $\beta(t)$ of the damped term in problem \eqref{mp} prevents the energy from increasing, meanwhile the negative part $\gamma(t)$ encourages the energy to grow in some extent. One can assume that if the negative part $\gamma(t)$ decrease to $0$ as time $t$ goes to $+\infty$ then the energy can hold bounded globally. That's the reason that we propose the assumption $\int_0^{+\infty}\gamma(\tau)\mathrm{d}\tau < +\infty$ in this paper in the first place. Moreover, with the assumption \eqref{bt0} we have:
\begin{equation*}
  \int_t^{t+1}\alpha(\tau)\mathrm{d}\tau > \int_t^{t+1}\gamma(\tau)\mathrm{d}\tau + 1,
\end{equation*}
\begin{equation*}
  \int_s^T\alpha(\tau)\mathrm{d}\tau > \int_s^T\gamma(\tau)\mathrm{d}\tau + [T-s],
\end{equation*}
where $[T-s]\in [T-s,T-s+1]$ is a integer. Take the limit $T\to+\infty$ in the above inequality, one can easily deduce that $\int_s^{+\infty}\alpha(\tau)\mathrm{d}\tau = +\infty$.

\end{remark}

\subsection{Bounded dissipativity}
\indent

We verified in the previous subsection that problem \eqref{mp} generates a process $\{S(t,s):t\geq s\}$ in $V=X^{\frac{1}{2}}\times L^2(\Omega)$, for any initial data $w(s)=w_0=(u_0,v_0)^T\in V$:
\begin{equation*}
  S(t,s)w_0 = w(t;w_0), \qquad t\geq0,
\end{equation*}
where $w(t)=(u(t),\partial_tu(t))^T\in C([s,T);V)$ denotes the unique global weak solution of \eqref{mp} corresponding to the initial data $w_0=(u_0,v_0)^T\in V$.

In what follows we will prove that the process $\{S(t,s):t\geq s\}$ is strongly pullback bounded dissipative(Definition \ref{sbd}). We start with a more explicit estimate for $\partial_tu(t)$.

\begin{lemma}\label{utd} Let the assumptions for Theorem \ref{etg} hold and let $(u(t),\partial_tu(t))$ be the solution corresponding to \eqref{mp} with initial data in a bounded domain, then there exists $T_0>0$ such that the following estimate holds for all $t-s>T_0$:
\begin{equation}\label{dut}
  \|\partial_tu(t)\|^2_{L^2(\Omega)} \leq R_0,
\end{equation}
where $R_0$ is independent of the norm of the initial data.
\end{lemma}
\begin{proof}Let $B\subset V$ be a bounded domain, $(u_0,v_0)=(u(s),v(s))\in B$. From the previous subsection we noticed that
\begin{equation}\label{ut55}
  \frac{1}{2}\|\partial_tu\|^2_{L^2(\Omega)}+ \int_s^t\beta(\tau)\|\partial_tu\|^2_{L^2(\Omega)}\mathrm{d}\tau \leq \frac{1}{2}\|\partial_tu(s)\|^2_{L^2(\Omega)} + \frac{1}{2}a(u(s),u(s)) - \int_{\Omega}F(u_0)\mathrm{d}x.
\end{equation}

Then take $g(\cdot):=\frac{1}{t-s}\big(a(u(s),u(s))-2\int_{\Omega}F(u_0)\mathrm{d}x\big)$ by applying Lemma \ref{G2} we have
\begin{equation}
   \begin{aligned}
      \|\partial_tu\|^2_{L^2(\Omega)} \leq& \|\partial_tu(s)\|^2_{L^2(\Omega)}\exp\{-\int_s^t\beta(\sigma)\mathrm{d}\sigma\}\\
       &+ \frac{1}{t-s}\Big(a(u(s),u(s))-2\int_{\Omega}F(u_0)\mathrm{d}x\Big)\int_s^t\exp\{{-\int_{\tau}^t\beta(\sigma)\mathrm{d}\sigma}\}\mathrm{d}\tau \\
     \leq& \|\partial_tu(s)\|^2_{L^2(\Omega)}e^{-(t-s)} + \frac{1}{t-s}\left(a(u(s),u(s))-2\int_{\Omega}F(u_0)\mathrm{d}x\right)\int_s^te^{-(t-\tau)}\mathrm{d}\tau \\
     \leq&  \|\partial_tu(s)\|^2_{L^2(\Omega)}e^{-(t-s)} + \frac{1}{t-s}\left(a(u(s),u(s))-2\int_{\Omega}F(u_0)\mathrm{d}x\right)(1-e^{-(t-s)})\\
     \leq&  C_0e^{-(t-s)} + \frac{C_0}{t-s}(1-e^{-(t-s)}),
   \end{aligned}
\end{equation}
where $C_0:=\max(\|\partial_tu(s)\|^2_{L^2(\Omega)},a(u(s),u(s))-2\int_{\Omega}F(u_0)\mathrm{d}x)$.

Note that
\begin{equation*}
  \lim_{t-s\to\infty}\big\{C_0e^{-(t-s)} + \frac{C_0}{t-s}(1-e^{-(t-s)})\big\}=0,
\end{equation*}
then for some $R_0>0$ small enough, there exists $T_0>0$ such that for all $t-s>T_0$ that \eqref{dut} holds.
\end{proof}

\begin{lemma}[\textbf{Pullback absorbing set}]\label{pa}Let $\Omega$ be  a bounded domain in $\mathbb{R}^N$ with smooth boundary. Then under the hypotheses \eqref{bt0}, \eqref{bt1}, \eqref{fcon} and \eqref{scon}, the process $\{S(t,s):t\geq s\}$ generated by problem \eqref{mp} is strongly bounded dissipative.
\end{lemma}
\begin{proof}Let $B\subset V$ be a bounded set.For each $w_0=(u_0,v_0)^T\in B$, let $w(t)=(u(t),\partial_tu(t))^T$ be the corresponding solution of \eqref{mp}. Consider the continuous functional $\Phi_\delta: X\to\mathbb{R}$ defined by
\begin{equation}\label{vf}
  \Phi_\delta(w) = \frac{1}{2}\|u\|^2_{X^{\frac{1}{2}}} + \frac{1}{2}\|\partial_tu\|_{L^2(\Omega)}^2 + \delta(t)(u,\partial_tu)-\int_{\Omega}F(u)\mathrm{d}x,
\end{equation}
where the constant $\delta>0$ and will be chosen appropriately later.

Note that it follows from \eqref{sconc} that for arbitrary $\epsilon\in\mathbb{R}^+$  and some constant $C_{\epsilon}>0$
\begin{eqnarray*}
  \frac{1}{2}\|u\|^2_{X^{\frac{1}{2}}} + \frac{1}{2}\|\partial_tu\|_{L^2(\Omega)}^2 &=& \Phi_\delta(w) - \delta(u,\partial_tu) + \int_{\Omega}F(u)\mathrm{d}x  \\
   &\leq& \Phi_\delta(w) + \frac{\delta}{\sqrt{\mu_1}}\|u\|_{X^{\frac{1}{2}}}\|\partial_tu\|_{L^2(\Omega)} + \frac{1}{2}c_0\|u(t)\|^2_{L^2(\Omega)} + c_1\int_{\Omega}|u(x,t)|\mathrm{d}x \\
   &\leq&  \Phi_\delta(w) + \frac{\delta}{2\mu_1}\|u\|^2_{X^{\frac{1}{2}}}+\frac{\delta}{2}\|\partial_tu\|^2_{L^2(\Omega)} + \frac{c_0+\epsilon}{2\mu_1}\|u\|^2_{X^{\frac{1}{2}}}  + C_{\epsilon}
\end{eqnarray*}
where we used Poincar\'{e}'s inequality and Young's inequality, thus,
\begin{equation*}
  \big(\frac{1}{2}-\frac{\delta}{2\mu_1}-\frac{c_0+\epsilon}{2\mu_1}\big)\|u\|^2_{X^{\frac{1}{2}}} + \big(\frac{1}{2} - \frac{\delta}{2}\big)\|\partial_tu\|^2_{L^2(\Omega)} \leq \Phi_\delta(w) + C_{\epsilon},
\end{equation*}
and we chose that $\delta<\min(1,\mu_1)$ and $\epsilon$ sufficiently small to imply that some constants $\bar{c}_0>0$ and $\bar{c}_1>0$ which are independent of time or initial data
\begin{equation}\label{vx}
  \|u\|^2_{X^{\frac{1}{2}}} + \|\partial_tu\|^2_{L^2(\Omega)} \leq \bar{c}_0\Phi_\delta(w) + \bar{c}_1.
\end{equation}

Due to the fact that $\gamma(t)$ is a continuous and integrable function, there exists a $0<M_{\gamma}<\infty$ such that for almost every $\tau\in(-\infty,\infty)$
   \begin{equation}\label{gamma_bdd}
      0\leq \gamma(\tau)\leq M_\gamma.
   \end{equation}

Now we choose some positive constant $C_{\gamma}+ C_{\beta}<M_\alpha<\infty$ and define that
\begin{equation}\label{A}
      A:=\{t\in\mathbb{R}|-M_\gamma\leq\beta(t)\leq M_{\alpha}\},
\end{equation}
   and
\begin{equation}\label{B}
      B:=\{t\in\mathbb{R}|\beta(s) > M_{\alpha}\}.
\end{equation}
According to our assumption \eqref{bt0} we can notice that $m(A)=\infty$. If not, we have \\ $\sum_{i=0}^{\infty}m(A\cap[i,i+1])<\infty$ and one can deduce that $m(A\cap[i,i+1])\to 0$ as $i\to\infty$, i.e., for $0<\epsilon<1$ there exists $n_A>0$ such that $m(A\cap[i,i+1])<\epsilon$ for $i\geq n_A$. Since $A\cup B = \mathbb{R}$, we have $m(B\cap[i,i+1])$ for $\geq 1-\epsilon$ for any $i\geq n_A$. Then for $t\geq n_A$ we can obtain \\$\int_t^{t+1}\beta(s)\mathrm{d}s \geq (1-\epsilon)M_{\alpha} - \epsilon M_{\gamma}>C_{\beta}$ which contradicts with our assumption \eqref{bt0}.

It is clear that from the equation \eqref{mp} we have
\begin{equation*}
   \begin{aligned}
     \frac{d}{dt}\Phi_\delta(w) =& a(u,\partial_tu) + (\partial_tu,\partial_{tt}u) - (f(u),\partial_tu) + \delta\|\partial_tu\|^2_{L^2(\Omega)} + \delta(u,\partial_{tt}u)\\
     =& (-\mathcal{L}u,\partial_t) + (\partial_tu,-\beta(t)\partial_tu+\mathcal{L}u+f(u)) - (f(u),\partial_tu) + \delta\|\partial_tu\|^2_{L^2(\Omega)}\\
     & + \delta(u,-\beta(t)\partial_tu + \mathcal{L}u + f(u)) \\
     =& -(\beta(t)-\delta)\|\partial_tu\|^2_{L^2(\Omega)} - \delta\|u\|^2_{X^{\frac{1}{2}}} - \delta\beta(t)(u,\partial_tu) + \delta\int_{\Omega}uf(u)\mathrm{d}x \\
     \leq&  -(\beta(t)-\delta)\|\partial_tu\|^2_{L^2(\Omega)} -\left(\delta-\frac{\delta(c_0+\epsilon)}{2\mu_1}\right)\|u\|^2_{X^{\frac{1}{2}}} +  \delta\beta(t)(u,\partial_tu) +C_{\epsilon},
   \end{aligned}
\end{equation*}
where we used the Cauchy-Schwartz inequality, Poincar\'{e}'s inequality and Young's inequality for any $0<\epsilon<\mu_1$. We can see from \eqref{sconc} that $c_0<\mu_1$, as long as $\epsilon<\mu_1$ in the above inequality we can guarantee that $\delta-\frac{\delta(c_0+\epsilon)}{2\mu_1}>0$.

Integrating the above equality over $[s,t]$ we have
\begin{equation*}
  \begin{aligned}
     \Phi_\delta(w(t)) \leq& \int_s^t\Big(-(\beta(\tau)-\delta)\|\partial_tu\|^2_{L^2(\Omega)} -(\delta-\frac{\delta(c_0+\epsilon)}{2\mu_1})\|u\|^2_{X^{\frac{1}{2}}} +  \delta\beta(\tau)(u,\partial_tu) +C_{\epsilon}\Big)\mathrm{d}\tau\\
     & + \Phi_\delta(w_0) \\
     \leq&  \int_s^t\Big((\gamma(\tau)+\delta)\|\partial_tu\|^2_{L^2(\Omega)} -(\delta-\frac{\delta(c_0+\epsilon)}{2\mu_1})\|u\|^2_{X^{\frac{1}{2}}} +  \delta\beta(\tau)(u,\partial_tu) +C_{\epsilon}\Big)\mathrm{d}\tau\\
     & + \Phi_\delta(w_0)
  \end{aligned}
\end{equation*}

Multiply $\chi_A$ on both sides of the above inequality:
\begin{equation*}
   \begin{aligned}
      \chi_A\Phi_\delta(w(t)) \leq& \int_{A\cap[s,t]}\Big((\gamma(\tau)+\delta)\|\partial_tu\|^2_{L^2(\Omega)} -(\delta-\frac{\delta(c_0+\epsilon)}{2\mu_1})\|u\|^2_{X^{\frac{1}{2}}} +  \delta\beta(\tau)(u,\partial_tu) +C_{\epsilon}\Big)\mathrm{d}\tau\\
      & + \Phi_\delta(w_0) \\
      \leq& \int_{A\cap[s,t]}\Big((\gamma(\tau)+\delta)\|\partial_tu\|^2_{L^2(\Omega)} -(\delta-\frac{\delta(c_0+\epsilon)}{2\mu_1})\|u\|^2_{X^{\frac{1}{2}}}\\
       &\hspace{20mm}+  \delta(M_{\gamma} + M_{\alpha})(u,\partial_tu) +C_{\epsilon}\Big)\mathrm{d}\tau + \Phi_\delta(w_0)  \\
     \leq&  \int_{A\cap[s,t]}\Big((\gamma(\tau) + \delta + \delta C_{\epsilon_1}(M_{\gamma}+M_{\alpha}))\|\partial_tu\|^2_{L^2(\Omega)}\\
     &\hspace{20mm}-(\delta-
     \frac{\delta(c_0+\epsilon)}{2\mu_1}-\frac{\delta\epsilon_1(M_{\gamma}+M_{\alpha})}{\mu_1})\|u\|^2_{X^{\frac{1}{2}}} + C_{\epsilon}\Big)\mathrm{d}\tau + \Phi_\delta(w_0),
   \end{aligned}
\end{equation*}
where we used Poincar\'{e}'s inequality and Cauchy-Schwarz inequality with $\epsilon>0$. We can guarantee $\delta-
   \frac{\delta(c_0+\epsilon)}{2\mu_1}-\frac{\delta\epsilon_1(M_{\gamma}+M_{\alpha})}{\mu_1}>0$ by choosing $0<\epsilon_1<\frac{\mu_1}{M_{\gamma}+M_{\alpha}}$.

Apply Lemma \ref{utd}, for any $t-s>T_0$ we have
\begin{equation*}
  \begin{aligned}
     \chi_A\Phi_\delta(w(t)) \leq& \int_{A\cap[s,t]}\Big((\gamma(\tau) + \delta + \delta C_{\epsilon_1}(M_{\gamma}+M_{\alpha}))R_0-(\delta-
     \frac{\delta(c_0+\epsilon)}{2\mu_1}-\frac{\delta\epsilon_1(M_{\gamma}+M_{\alpha})}{\mu_1})\|u\|^2_{X^{\frac{1}{2}}}\\
      &\hspace{30mm}+ C_{\epsilon}\Big)\mathrm{d}\tau + \Phi_\delta(w_0)  \\
     \leq& r_0R_0 + \int_{A\cap[s,t]}\Big((\delta + \delta C_{\epsilon_1}(M_{\gamma}+M_{\alpha}))R_0-(\delta-
     \frac{\delta(c_0+\epsilon)}{2\mu_1}-\frac{\delta\epsilon_1(M_{\gamma}+M_{\alpha})}{\mu_1})\|u\|^2_{X^{\frac{1}{2}}}\\
      &\hspace{30mm}+ C_{\epsilon}\Big)\mathrm{d}\tau + \Phi_\delta(w_0),
  \end{aligned}
\end{equation*}
where $r_0 := \int_{-\infty}^{\infty}\gamma(s)\mathrm{d}s$.

Combining with the inequality \eqref{vx} implies that for $$\|u\|^2_{X^{\frac{1}{2}}}\geq R_1 > \frac{2\mu_1\delta R_0^2 + C_{\epsilon_1}\delta(M_{\gamma}+M_{\alpha})R_0 + 2\mu_1C_{\epsilon}}{\delta\mu_1 -\delta(c_0+\epsilon)-2\epsilon_1\delta(M_{\gamma}+M_{\alpha})} $$ we have
\begin{equation}\label{A2}
  \chi_A(\|\partial_tu\|^2_{L^2(\Omega)} + \|u\|^2_{X^{\frac{1}{2}}}) \leq \bar{c}_0(-C_{\delta}m(A\cap[s,t])+ r_0R_0 + \Phi_\delta(w_0))+ \bar{c}_1,
\end{equation}
where $C_{\delta}>0$.

On the other hand, for $\delta=0$ we have
\begin{equation*}
  \frac{d}{dt}\Phi_0(w(t)) = \beta(t)\|\partial_tu\|^2_{L^2(\Omega)}.
\end{equation*}

Integrate the above equality on $[s,t]$,
\begin{equation}\label{B1}
  \Phi_0(w(t))=-\int_s^t\beta(\tau)\|\partial_tu\|^2_{L^2(\Omega)}\mathrm{d}\tau + \Phi_0(w_0),
\end{equation}
then we multiply $\chi_B$ on both sides of \eqref{B1}
\begin{eqnarray*}
  \chi_B\Phi_0(w(t)) &=& -\int_{B\cap[s,t]}\beta(\tau)\|\partial_tu\|^2_{L^2(\Omega)}\mathrm{d}\tau + \Phi_0(w_0) \\
   &\leq&  -\int_{B\cap[s,t]}M_{\alpha}\|\partial_tu\|^2_{L^2(\Omega)}\mathrm{d}\tau + \Phi_0(w_0).
\end{eqnarray*}

If $\|\partial_tu\|^2_{L^2(\Omega)} \geq \frac{1}{2}R_0$, then
\begin{equation}\label{B2}
  \chi_B\Phi_0(w(t)) \leq -M_{\alpha}R_0m(B\cap[s,t]) + 2\Phi_0(w_0),
\end{equation}
also with the same method for \eqref{vx} we obtain
\begin{equation}\label{vxx}
  \|u\|^2_{X^{\frac{1}{2}}} + \|\partial_tu\|^2_{L^2(\Omega)} \leq \tilde{c}_0\Phi_0(w) + \tilde{c}_1.
\end{equation}

Then we combine \eqref{A2}, \eqref{B2} and \eqref{vxx} we have
\begin{equation*}
   \begin{aligned}
      \|\partial_tu\|^2_{L^2(\Omega)} + \|u\|^2_{X^{\frac{1}{2}}} \leq& \bar{c}_0(-C_{\delta}m(A\cap[s,t])+ r_0R_0 + \Phi_{\delta}(w_0))+ \bar{c}_1\\
      & \tilde{c}_0(-M_{\alpha}R_0m(B\cap[s,t]) + 2\Phi(w_0)) + \tilde{c}_1.
   \end{aligned}
\end{equation*}

Therefore, since we proved earlier that $m(A)=\infty$, there exists $T_A>0$ such that
\begin{equation}\label{A1}
  \|\partial_tu\|^2_{L^2(\Omega)} + \|u\|^2_{X^{\frac{1}{2}}} \leq R_0 + \min(R_1,\bar{c}_1+\tilde{c}_1)
\end{equation}
for all $t-s>\max(T_0,T_A)$.

This proves that $\{S(t,s):t\geq s\}$ is strongly dissipative in the sense of Definition \ref{sbd}, i.e. for any bounded set $B\subset V$ there exists an $R>0$ and a time $T(B)$ such that
 \begin{equation}\label{abb}
   S(t+s,s)B \subset B_V(0,R)\qquad \text{for all}\quad t\geq T(B),
 \end{equation}
uniformly for all $s\in\mathbb{R}$.
\end{proof}

\section{Pullback attractor}
\noindent

The main result of the present paper is the following theorem.
\begin{theorem}\label{mr}Let the non-linearity $f$ satisfies the growth restriction \eqref{fcon} and dissipativity assumption \eqref{scon}. Assume also that the coefficient of damping term satisfies the conditions \eqref{bt0} and \eqref{bt1}. Then \eqref{mp} has a pullback attractor $\mathcal{A}(\cdot)$.
\end{theorem}

In this section we recall a well-known compactness criterion established by Chuesshov and Lasiecka \cite{II, II1} for autonomous systems. Non-autonomous version of the results were presented in \cite{CF, MF}, one recast the non-autonomous equation by the skew-product system and the other considered the problem as a process. The following compactness criterion for non-autonomous system we presented is similar to the Theorem 3.2 in \cite{MF} with $X=X_t$ and we omit the proof here.

\begin{definition}[\cite{CF}] Let $X$ be a Banach space and $B$ be a bounded subset of $X$. We call a function $\psi(\cdot,\cdot)$, defined on $X\times X$, a contractive function on $B\times B$ if for any sequence $\{x_n\}_{n=1}^{\infty}\subset B$, there is a subsequence $\{x_{n_{k}}\}_{k=1}^{\infty}\subset\{x_n\}_{n=1}^{\infty}$ such that
\begin{equation*}
  \lim_{k\to\infty}\lim_{l\to\infty}\psi(x_{n_{k}},x_{n_{l}}) = 0.
\end{equation*}
\end{definition}

\begin{theorem}\label{cptness}Let $X$ be a Banach spaces and $S(t,s):X\to X$ be an evolution process that possesses a pullback absorbing family $\hat{B}_0=\{B_0(s)\}_{s\in\mathbb{R}}$. Suppose that for any $t\in\mathbb{R}$ and $\epsilon>0$ there exists a time $s_{\epsilon}<t$ and a contractive function $\psi_{\epsilon}:B_0(s_{\epsilon})\times B_0(s_{\epsilon})\to\mathbb{R}$, such that
\begin{equation*}
  \|S(t,s_{\epsilon})x-S(t,s_{\epsilon})y\|_X \leq \epsilon + \psi_{\epsilon}(x,y).
\end{equation*}
Then the process is pullback asymptotically compact.
\end{theorem}

\begin{lemma}\label{ac}Under the assumptions of Theorem \ref{pa} the corresponding process is pullback asymptotically compact.
\end{lemma}

%

The idea of the proof is similar to that in \cite{II,II1}, and also in \cite{CF} for non-autonomous case. In order to prove pullback asymptotically compact for the process $S(t,s)$ generated by the problem \eqref{mp}, we start to derive some energy inequalities.

Let $B$ is a bounded subset of $V$, assume $u_1(t)$ and $u_2(t)$ are  weak solutions of \eqref{mp} with initial data $u_1(s)=u_0^1$, $\partial_tu_1(s)=u_1^1$, $u_2(s)=u_0^2$ and $\partial_tu_2(s)=u_1^2$, where $(u_0^1,u_1^1),~(u_0^2,u_1^2)\in B$.
Then we observe that $w=u_1-u_2$ is a weak solution of
\begin{equation}\label{ab}
  \partial_{tt}w + \beta(t)w_t = \mathcal{L}w + f(u_1) - f(u_2), \qquad x\in\Omega,~ t\geq s,
\end{equation}
with Dirichlet boundary condition and initial conditions
\begin{equation*}
  w(0) = u_0^1-u_0^2\quad\text{and}\quad \partial_{t}w(0)=u_1^1-u_1^2.
\end{equation*}

First multiply equation \eqref{ab} by $\partial_tw$ and integrate over $\Omega$, we obtain
\begin{equation}\label{3}
  \frac{1}{2}\frac{d}{dt}\|\partial_tw\|^2_{L^2(\Omega)} + \frac{1}{2}\frac{d}{dt}\|w\|^2_{X^{\frac{1}{2}}} + \beta(t)\|\partial_tw\|^2_{L^2(\Omega)} = (f(u_1)-f(u_2),\partial_tw),
\end{equation}
integrating \eqref{3} over $[s,t]$ with $s< t$, we have that
\begin{equation}\label{Gr1}
  \begin{aligned}
     \|\partial_tw\|^2_{L^2(\Omega)} + 2\int_s^t\beta(\tau)\|\partial_tw\|^2_{L^2(\Omega)}\mathrm{d}\tau \leq & 2\int_s^t(f(u_1)-f(u_2),\partial_tw)\mathrm{d}\tau + \|\partial_tw(s)\|^2_{L^2(\Omega)}\\
      &+ \|w(s)\|^2_{X^{\frac{1}{2}}}.
  \end{aligned}
\end{equation}

We apply the extension Gronwall Lemma \ref{G2} with
\begin{equation*}
  g(\tau) = 2(f(u_1)-f(u_2),\partial_tw(\tau)) + \frac{1}{t-s}\|w(s)\|^2_{X^{\frac{1}{2}}},
\end{equation*}
thus the following estimate hold
\begin{equation}\label{Gr2}
  \begin{aligned}
     \|\partial_tw(t)\|^2_{L^2(\Omega)} \leq & \|\partial_tw(s)\|^2_{L^2(\Omega)}\exp\left\{-2\int_s^t\beta(\sigma)\mathrm{d}\sigma\right\}\\
      &+ 2\int_s^t(f(u_1)-f(u_2),\partial_tw(\tau))\exp\left\{-2\int_\tau^t\beta(\sigma)\mathrm{d}\sigma\right\}\mathrm{d}\tau\\
      &+ \int_s^t\frac{1}{t-s}\|w(s)\|^2_{X^{\frac{1}{2}}}\exp\left\{-2\int_\tau^t\beta(\sigma)\mathrm{d}\sigma\right\}\mathrm{d}\tau.
  \end{aligned}
\end{equation}

It follows from Remark 3.4 that
\begin{equation}\label{Gr3}
   \begin{aligned}
      \|\partial_tw(t)\|^2_{L^2(\Omega)} \leq &\|\partial_tw(s)\|^2_{L^2(\Omega)}e^{-2(t-s)} + 2\int_s^t(f(u_1)-f(u_2),\partial_tw(\tau))e^{-2(t-\tau)}\mathrm{d}\tau\\
       &+ \frac{1}{t-s}\|w(s)\|^2_{X^{\frac{1}{2}}}\int_s^te^{-2(t-\tau)}\mathrm{d}\tau.
   \end{aligned}
\end{equation}

Next we construct a function which we will prove that it is a contractive function.  Define an energy functional
\begin{equation*}
  E_w(t) = \frac{1}{2}\|\partial_tw\|^2_{L^2(\Omega)} + \frac{1}{2}\|w\|^2_{X^{\frac{1}{2}}}.
\end{equation*}
integrating \eqref{3} over $[t,T]$ with $t\leq T$, we have that
\begin{equation}\label{4}
  E_w(T) + \int_t^T\beta(\tau)\|\partial_tw\|^2_{L^2(\Omega)}\mathrm{d}\tau = \int_t^T(f(u_1)-f(u_2),\partial_tw)\mathrm{d}\tau + E_w(t),
\end{equation}
where $(\cdot,\cdot)$ denotes the inner product in $L^2(\Omega)$.

And integrating \eqref{4} over $[s,T]$ with respect to $t$ we obtain
\begin{equation}\label{5}
  \begin{aligned}
     &(T-s)E_w(T)+ \int_s^T\int_t^T\beta(\tau)\|\partial_tw\|^2_{L^2(\Omega)}\mathrm{d}\tau\mathrm{d}t\\
     \leq &\int_s^T\int_t^T(f(u_1)-f(u_2),\partial_tw)\mathrm{d}\tau\mathrm{d}t + \int_s^TE_w(t)\mathrm{d}t.
  \end{aligned}
\end{equation}


Multiplying equation \eqref{ab} by $w$ and integrating over $\Omega$, we obtain
\begin{equation}\label{6}
  \frac{d}{dt}(\partial_tw,w) - \|\partial_tw\|^2_{L^2(\Omega)} + \beta(t)(\partial_tw,w) + \|w\|^2_{X^{\frac{1}{2}}} = (f(u_1)-f(u_2),w),
\end{equation}
then integrating \eqref{6} over $[s,T]$, we have
\begin{equation*}
  \begin{aligned}
     &\int_s^T\beta(t)(\partial_tw,w)\mathrm{d}t + (\partial_tw(T),w(T)) + \int_s^T\|w\|^2_{X^{\frac{1}{2}}}\mathrm{d}t\\
     = &\int_s^T\|\partial_tw\|^2_{L^2(\Omega)}\mathrm{d}t + (\partial_tw(s),w(s)) +\int_s^T(f(u_1)-f(u_2),w)\mathrm{d}t,
  \end{aligned}
\end{equation*}
therefore we deduce that
\begin{equation}\label{66}
  \begin{aligned}
     \int_s^T\|w\|^2_{X^{\frac{1}{2}}}\mathrm{d}t = &\int_s^T\|\partial_tw\|^2_{L^2(\Omega)}\mathrm{d}t + (\partial_tw(s),w(s)) - (\partial_tw(T),w(T)) \\
  &+\int_s^T(f(u_1)-f(u_2),w)\mathrm{d}t - \int_s^T\beta(t)(\partial_tw,w)\mathrm{d}t
  \end{aligned}
\end{equation}

Combine \eqref{5} and \eqref{66} we can obtain
\begin{equation*}
  \begin{aligned}
     &(T-s)E_w(T)+ \int_s^T\int_t^T\beta(\tau)\|\partial_tw\|^2_{L^2(\Omega)}\mathrm{d}\tau\mathrm{d}t\\
     \leq &\int_s^T\int_t^T(f(u_1)-f(u_2),\partial_tw)\mathrm{d}\tau\mathrm{d}t + \int_s^T\|\partial_tw\|^2_{L^2(\Omega)}\mathrm{d}t + \frac{1}{2}(\partial_tw(s),w(s))\\
     &-\frac{1}{2}(\partial_tw(T),w(T)) + \frac{1}{2}\int_s^T(f(u_1)-f(u_2),w)\mathrm{d}t - \frac{1}{2}\int_s^T\beta(t)(\partial_tw,w)\mathrm{d}t,
  \end{aligned}
\end{equation*}
hence
\begin{equation}\label{T1}
   \begin{aligned}
      (T-s)E_w(T)
     \leq &\int_s^T\int_t^T\gamma(\tau)\|\partial_tw\|^2_{L^2(\Omega)}\mathrm{d}\tau\mathrm{d}t + \int_s^T\int_t^T(f(u_1)-f(u_2),\partial_tw)\mathrm{d}\tau\mathrm{d}t\\
      &+ \int_s^T\|\partial_tw\|^2_{L^2(\Omega)}\mathrm{d}t
      + \frac{1}{2}(\partial_tw(s),w(s))-(\partial_tw(T),w(T))\\
      &+ \frac{1}{2}\int_s^T(f(u_1)-f(u_2),w)\mathrm{d}t - \frac{1}{2}\int_s^T\beta(t)(\partial_tw,w)\mathrm{d}t.
   \end{aligned}
\end{equation}

We can infer from \eqref{Gr3} that
\begin{equation}\label{wt3}
  \begin{aligned}
     \int_s^T\|\partial_tw\|^2_{L^2(\Omega)}\mathrm{d}t \leq &\|\partial_tw(s)\|^2_{L^2(\Omega)}\int_s^Te^{-2(t-s)}\mathrm{d}t+ \|w(s)\|^2_{X^{\frac{1}{2}}}\int_s^T\int_s^t\frac{1}{t-s}e^{-2(t-\tau)}\mathrm{d}\tau\mathrm{d}t\\ &+2\int_s^T\int_s^t(f(u_1)-f(u_2),\partial_tw(\tau))e^{-2(t-\tau)}\mathrm{d}\tau\mathrm{d}t,
  \end{aligned}
\end{equation}
and
\begin{equation}\label{wt2}
   \begin{aligned}
      &\int_s^T\int_t^T\gamma(\tau)\|\partial_tw\|^2_{L^2(\Omega)}\mathrm{d}\tau\mathrm{d}t\\
      \leq& \|\partial_tw(s)\|^2_{L^2(\Omega)} \int_s^T\int_t^T\gamma(\tau)e^{-2(\tau-s)}\mathrm{d}\tau\mathrm{d}t+
       \|w(s)\|^2_{X^{\frac{1}{2}}} \int_s^T\int_t^T\gamma(\tau)\frac{1}{\tau-s}\int_s^{\tau}e^{-2(\tau-\sigma)}\mathrm{d}\sigma\mathrm{d}\tau\mathrm{d}t\\
      &+ \int_s^T\int_t^T\gamma(\tau)\int_s^{\tau}(f(u_1)-f(u_2),\partial_tw(\sigma))e^{-2(\tau-\sigma)}\mathrm{d}\sigma\mathrm{d}\tau\mathrm{d}t, \\
      \leq& \|\partial_tw(s)\|^2_{L^2(\Omega)} \int_s^T e^{-2(t-s)}\int_t^T\gamma(\tau)\mathrm{d}\tau\mathrm{d}t + \|w(s)\|^2_{X^{\frac{1}{2}}} \int_s^T\int_t^T\gamma(\tau)\frac{1}{\tau-s}\int_s^{\tau}e^{-2(\tau-\sigma)}\mathrm{d}\sigma\mathrm{d}\tau\mathrm{d}t\\
      &+  \int_s^T\int_t^T\gamma(\tau)\int_s^{\tau}(f(u_1)-f(u_2),\partial_tw(\sigma))e^{-2(\tau-\sigma)}\mathrm{d}\sigma\mathrm{d}\tau\mathrm{d}t\\
      \leq&  C_{\gamma,B} \int_s^T e^{-2(t-s)}\mathrm{d}t + \int_s^T\int_t^T\gamma(\tau)\int_s^{\tau}(f(u_1)-f(u_2),\partial_tw(\sigma))e^{-2(\tau-\sigma)}\mathrm{d}\sigma\mathrm{d}\tau\mathrm{d}t \\
      &+ \|w(s)\|^2_{X^{\frac{1}{2}}} \int_s^T\int_t^T\gamma(\tau)\frac{1}{\tau-s}\int_s^{\tau}e^{-2(\tau-\sigma)}\mathrm{d}\sigma\mathrm{d}\tau\mathrm{d}t,\\
   \end{aligned}
\end{equation}
where $C_{\gamma,B}$ is a constant depends on $\int_{-\infty}^{\infty}\gamma(s)\mathrm{d}s$ and the radius of the bounded domain $B$.

Thus from \eqref{T1}, \eqref{wt3} and \eqref{wt2} we observe that
\begin{equation}\label{TT1}
   \begin{aligned}
     (T-s)E_w(T)
     \leq &C_{\gamma,B} \int_s^T e^{-2(t-s)}\mathrm{d}t + \int_s^T\int_t^T\gamma(\tau)\int_s^{\tau}(f(u_1)-f(u_2),\partial_tw(\sigma))e^{-2(\tau-\sigma)}\mathrm{d}\sigma\mathrm{d}\tau\mathrm{d}t \\
     &+ \|w(s)\|^2_{X^{\frac{1}{2}}} \int_s^T\int_t^T\gamma(\tau)\frac{1}{\tau-s}\int_s^{\tau}e^{-2(\tau-\sigma)}\mathrm{d}\sigma\mathrm{d}\tau\mathrm{d}t\\
      &+ \int_s^T\int_t^T(f(u_1)-f(u_2),\partial_tw)\mathrm{d}\tau\mathrm{d}t
     + \|\partial_tw(s)\|^2_{L^2(\Omega)}\int_s^Te^{-2(t-s)}\mathrm{d}t\\
     &+ 2\int_s^T\int_s^t(f(u_1)-f(u_2),\partial_tw(\tau))e^{-2(t-\tau)}\mathrm{d}\tau\mathrm{d}t \\
     &+ \|w(s)\|^2_{X^{\frac{1}{2}}}\int_s^T\int_s^t\frac{1}{t-s}e^{-2(t-\tau)}\mathrm{d}\tau\mathrm{d}t + \frac{1}{2}(\partial_tw(s),w(s))-(\partial_tw(T),w(T))\\
     &+ \frac{1}{2}\int_s^T(f(u_1)-f(u_2),w)\mathrm{d}t - \frac{1}{2}\int_s^T\beta(t)(\partial_tw,w)\mathrm{d}t.
   \end{aligned}
\end{equation}

Define
\begin{equation}\label{T2}
  \begin{aligned}
     \psi_{T,s}(u_1,u_2) =& \frac{1}{T-s}\bigg[\int_s^T\int_t^T(f(u_1)-f(u_2),\partial_tw)\mathrm{d}\tau\mathrm{d}t \\
     &  + \frac{1}{2}\int_s^T(f(u_1)-f(u_2),w)\mathrm{d}t - \frac{1}{2}\int_s^T\beta(t)(\partial_tw,w)\mathrm{d}t\\
     &+ \int_s^T\int_t^T\gamma(\tau)\int_s^{\tau}(f(u_1)-f(u_2),\partial_tw(\sigma))e^{-2(\tau-\sigma)}\mathrm{d}\sigma\mathrm{d}\tau\mathrm{d}t \\
     &+ 2\int_s^T\int_s^t(f(u_1)-f(u_2),\partial_tw(\tau))e^{-2(t-\tau)}\mathrm{d}\tau\mathrm{d}t\bigg],
  \end{aligned}
\end{equation}
then we have
\begin{equation}\label{T3}
   \begin{aligned}
      E_w(T) \leq &\frac{1}{T-s}\bigg[\frac{1}{2}(\partial_tw(s),w(s))-\frac{1}{2}(\partial_tw(T),w(T)) +   \|\partial_tw(s)\|^2_{L^2(\Omega)}\int_s^Te^{-2(t-s)}\mathrm{d}t\\
      &+ \|w(s)\|^2_{X^{\frac{1}{2}}}\int_s^T\int_s^t\frac{1}{t-s}e^{-2(t-\tau)}\mathrm{d}\tau\mathrm{d}t + C_{\gamma,B} \int_s^T e^{-2(t-s)}\mathrm{d}t\\
      &+ \|w(s)\|^2_{X^{\frac{1}{2}}} \int_s^T\int_t^T\gamma(\tau)\frac{1}{\tau-s}\int_s^{\tau}e^{-2(\tau-\sigma)}\mathrm{d}\sigma\mathrm{d}\tau\mathrm{d}t\bigg] + \psi_{T,s}(u_1,u_2).
   \end{aligned}
\end{equation}

With the above inequalities, we are now ready to prove the pullback asymptotic compactness i.e., Lemma \ref{ac}.

\begin{proof}  For any $\epsilon>0$, we can see from \eqref{T3} if we take $-s$ large enough and by virtue of the dominate convergence theorem, there exists
\begin{equation}\label{14}
  E_w(T) \leq \epsilon + \psi_{T,s}(u_1,u_2),
\end{equation}
where $u_1,~u_2$ belong to a bounded subset in $V = X^{\frac{1}{2}}\times L^2(\Omega)$ since we proved the well-posedness of the solutions of \eqref{mp}  for all $t\in\mathbb{R}$.

Therefore according to the Theorem \ref{cptness}, only if the function $\psi_{T,s}(\cdot,\cdot)$ is a contractive function then the proof of Lemma \ref{ac} is complete. Let $B\in V$ be a bounded set and $(u_n,\partial_tu_n)$ be the corresponding solutions of $(u_0^n,\partial_tu_0^n)\in B$ for the problem \eqref{mp}, $n=1,2,\cdots$. The rest work is to prove $\lim_{n\to\infty}\lim_{m\to\infty}\psi_{T,s}(u_n,u_m)=0$, the proof is similar to the one in \cite{CF}.

According to Sect.3 $(u_n,\partial_tu_n)$ is bounded in $X^{\frac{1}{2}}\times L^2(\Omega)$. Note that the compact embedding $X^{\frac{1}{2}}\hookrightarrow L^p(\Omega)$ is valid for $1\leq p<\frac{2Q}{Q-2}$, without loss of generality we assume that for any $s\leq t\in\mathbb{R}$
\begin{eqnarray}\
  \label{u_star_weak_convergence}u_n \to& u  \qquad &\ast\text{-weakly in}~ L^{\infty}(s,t;X^{\frac{1}{2}}), \\
  \label{up}u_n \to& u \qquad &\text{in}\quad L^{p+1}(s,t;L^{p+1}(\Omega)), \\
  \label{u_t_star_weak_convergence} \partial_t u_n \to& \partial_t u \qquad &\ast\text{-weakly in}~ L^{\infty}(s,t;L^2(\Omega)), \\
  \partial_t u_n \to& \partial_t u \qquad &\text{weakly in}~L^2(\Omega), \\
  u_n \to& u \qquad &\text{in}~ L^2(s,t;L^2(\Omega)).
\end{eqnarray}

Now we deal with each term in \eqref{T2}.

First, applying growth condition \eqref{fcon} and \eqref{up} we can obtain
\begin{equation*}
   \lim_{n\to\infty}\lim_{m\to\infty}\int_s^T\big(f(u_n)-f(u_m),u_n-u_m\big)\mathrm{d}t = 0.
\end{equation*}

Second, we derived that $f\in L^2(\Omega)$ in Sect.3, note that
\begin{equation*}
   \begin{aligned}
      &\int^T_t\int_{\Omega}(\partial_tu_{n}(\tau)-\partial_tu_m(\tau))(f(u_n)-f(u_m))\mathrm{d}x\mathrm{d}\tau\\
      =&\int^T_t\int_{\Omega}(\partial_tu_n(\tau))f(u_n)\mathrm{d}x\mathrm{d}\tau+\int^T_t\int_{\Omega}\partial_tu_mf(u_m)\mathrm{d}x\mathrm{d}\tau
      -\int^T_t\int_{\Omega}\partial_tu_n(\tau)f(u_m)\mathrm{d}x\mathrm{d}\tau\\
      &-\int^T_t\int_{\Omega}\partial_tu_m(\tau)f(u_n)\mathrm{d}x\mathrm{d}\tau\\
      =&\int_{\Omega}F(u_n(T))-\int_{\Omega}F(u_n(t))+\int_{\Omega}F(u_m(T))-\int_{\Omega}F(u_m(t))\\ &-\int^T_t\int_{\Omega}\partial_tu_n(\tau)f(u_m)\mathrm{d}x\mathrm{d}\tau-\int^T_t\int_{\Omega}\partial_tu_m(\tau)f(u_n)\mathrm{d}x\mathrm{d}\tau
   \end{aligned}
\end{equation*}
combining with \eqref{u_star_weak_convergence}, \eqref{u_t_star_weak_convergence}, take $m\to\infty$ and $n\to\infty$, we get
\begin{equation}\label{ft1}
   \begin{aligned}
      &\lim_{n\to\infty}\lim_{m\to\infty}\int^T_t\int_{\Omega}(\partial_tu_{n}(\tau)-\partial_tu_m(\tau))(f(u_n)-f(u_m))\mathrm{d}x\mathrm{d}\tau\\
      =&\int_{\Omega}F(u(T))-\int_{\Omega}F(u(t))+\int_{\Omega}F(u(T))-\int_{\Omega}F(u(t))\\
      &-\int^T_t\int_{\Omega}\partial_tu(\tau)f(u)\mathrm{d}x\mathrm{d}\tau-\int^T_t\int_{\Omega}\partial_tu(\tau)f(u)\mathrm{d}x\mathrm{d}\tau\\
      =& 0.
   \end{aligned}
\end{equation}

Third, due to the fact that $\beta(t)$ is a continuous function and apply (4.14) and (4.15) we have
\begin{equation*}
  \lim_{n\to\infty}\lim_{m\to\infty}\int^T_t\beta(\tau)\big(\partial_tu_n-\partial_tu_m,u_n-u_m\big)\mathrm{d}\tau =0.
\end{equation*}

To this end, by virtue of equality \eqref{ft1}, through taking $n\to\infty$ and $m\to\infty$ we can obtain
\begin{equation*}
  \lim_{n\to\infty}\lim_{m\to\infty}\int_t^T\int_s^{\tau}\gamma(\tau)(f(u_n)-f(u_m),\partial_tw(\sigma))e^{-2(\tau-\sigma)}\mathrm{d}\sigma\mathrm{d}\tau = 0,
\end{equation*}
and
\begin{equation*}
  \lim_{n\to\infty}\lim_{m\to\infty} \int_s^t(f(u_1)-f(u_2),\partial_tw(\tau))e^{-2(t-\tau)}\mathrm{d}\tau=0.
\end{equation*}

Therefore, we verified that $\psi_{T,s}(\cdot,\cdot)$ is a contractive function.
\end{proof}

Now we can complete the proof of our main result Theorem $\ref{mr}$.
\begin{proof} We can see from Lemma \ref{pa} Lemma \ref{ac} the conditions for Theorem \ref{eth} are all satisfied respectively. Hence there exists a pullback attractor for the process $S(t,s)$ generated by the problem \eqref{mp}.
\end{proof}

\begin{remark}The argument will be no difference if we consider the nonlinear term $f(u)$ as $f(u,t)$. Of course, there will need certain regular assumptions for $f(u,t)$ with respect to $t$.
\end{remark}

\begin{remark} In particular, let the assumptions \eqref{fcon}, \eqref{scon}, \eqref{bt0} and \eqref{bt1} hold, one can obtain the existence and finite fractal dimension of the global attractor for the semigroup generated by the problem \eqref{mp} with $-\mathcal{L}=-\triangle$.
\end{remark}

\begin{remark}Actually, one can consider more general assumptions for $\beta(t)$ than our case. For instance, we may propose a reasonable hypothesis that the positive part $\alpha(t)$ of $\beta(t)$ can be split into two parts as $\alpha(t) = \alpha_1(t) + \alpha_2(t)$ such that
\begin{equation*}
  \int_0^t\big(\alpha_2(s)-\gamma(s)\big)\mathrm{d}s \to 0 \quad\text{as}\quad t\to\infty,
\end{equation*}
where $\gamma(t)$ is the negative part of $\beta(t)$.
And we may conjecture that the solutions of \eqref{mp} exist globally since the decreasing factor $\alpha_2(t)$ cancels the increasing influence of $\gamma(t)$ out and the left part $\alpha_1(t)$ guarantee energy decays in infinity.
\end{remark}

\textbf{Acknowledgments.} During the preparation of this paper, we benefited of discussions with many people, in particular we wish to thank Prof. Wenxian Shen for many discussions and comments.
We were supported by the National Nature Science Foundation of China grants 11471148 and 11522109, and also by Fundamental Research Funds for the Central Universities No. lzujbky-2017-it52 and No. lzujbky-2017-160.

\end{document}